\documentclass[11pt]{amsart}
\usepackage{amsmath,amsthm,amsfonts,amssymb,times}

\usepackage[notref,notcite]{showkeys}   

\usepackage{titlesec}
\titleformat{\section}  
{\normalfont\Large\bfseries}
{\thesection}{1em}{}
\titleformat{\subsection}
{\normalfont\large\bfseries}
{\thesubsection}{1em}{}

\numberwithin{equation}{section}  

\DeclareMathAlphabet{\curly}{U}{rsfs}{m}{n}  

\theoremstyle{remark}

\theoremstyle{plain}

\newtheorem{lem}{Lemma}[section]
\newtheorem{thm}{Theorem}

\numberwithin{equation}{section}

\newcommand{\NN}{{\mathbb N}}

\makeatletter
\renewcommand{\pmod}[1]{\allowbreak\mkern7mu({\operator@font mod}\,\,#1)}
\makeatother

\newcommand{\bal}{\[\begin{aligned}}
\newcommand{\eal}{\end{aligned}\]}

\newcommand{\be}{\begin{equation}}
\newcommand{\ee}{\end{equation}}

\newcommand{\ssum}[1]{\sum_{\substack{#1}}}  

\newcommand{\lam}{\ensuremath{\lambda}}

\renewcommand{\a}{\ensuremath{\alpha}}

\newcommand{\g}{\gamma}

\newcommand{\eps}{\ensuremath{\varepsilon}}

\renewcommand{\le}{\leqslant}

\renewcommand{\ge}{\geqslant}

\newcommand{\fl}[1]{{\ensuremath{\left\lfloor {#1} \right\rfloor}}}

\newcommand{\order}{\asymp}      
\renewcommand{\(}{\left(}
\renewcommand{\)}{\right)}

\newcommand{\pfrac}[2]{\left(\frac{#1}{#2}\right)}  

\begin{document}

\title{Integers divisible by a large shifted prime}
\author{Kevin Ford}
\date{}
\address{Department of Mathematics, 1409 West Green Street, University
of Illinois at Urbana-Champaign, Urbana, IL 61801, USA}
\email{ford@math.uiuc.edu}

\begin{abstract}
We determine the exact order of growth of $N(x,y)$, the number of integers $n\le x$ divisible by 
a shifted prime $p-1>y$, uniformly for all $x\ge 2y\ge 4$.
\end{abstract}

\thanks{
 Research supported in part by NSF grant DMS-1501982.}

\maketitle


\section{Introduction}

Let $N(x,y)$ be the number of integers $n\le x$ divisible by some number $p-1$, where $p>y$ is prime.
The problem of bounding $N(x,y)$ originated in 1980 with Erd\H os and Wagstaff \cite{EW}, who proved the upper bound
\be\label{EW}
N(x,y) \ll \frac{x}{(\log y)^c}, \qquad \text{some constant }c>0,
\ee
uniformly for $x>y>10$, and applied this estimate to the study of denominators of Bernoulli numbers.

In \cite{MPP}, the following improved estimates were shown.  Here $\log_2 x=\log\log x$, 
$\log_3 x =\log\log\log x$ and $\delta=1-\frac{1+\log\log 2}{\log 2} = 0.08607\ldots$.

\newtheorem*{thmA}{Theorem A}

\begin{thmA}[\cite{MPP}]\label{MPP}
(i) If $3\le y\le x$, then
 \[
   N(x,y) \ll \frac{x}{(\log y)^\delta (\log_2 y)^{1/2}},
 \]
and for every $\eps>0$ there is an $\eta>0$ so that for $3\le y\le x\exp\{-(\log x)^{1-\eta} \}$,
\[
 N(x,y) \gg \frac{x}{(\log y)^{\delta+\eps}}.
\]

(ii) If $y=x/\exp\{ (\log x)^\alpha \}$, and $\frac{1}{\log 4} \le \alpha \le 1 - \frac{\log_3 x}{\log_2 x}$,
then
\[
 N(x,y) = \frac{x(\log_2 x)^{O(1)}}{(\log x)^{\delta+\alpha-1-(\log \alpha)/\log 2}}.
\]

(iii)  If $y=x/\exp\{ (\log x)^\alpha \}$, and $0 \le \alpha \le \frac{1}{\log 4}$, then
\[
 N(x,y) = \frac{x \log(x/y) (\log_2 (x/y))^{O(1)}}{\log x}.
\]
\end{thmA}

The authors remark (Remark 2.11 of \cite{MPP}) that they can very easily establish the following with their methods:
For any $\eps>0$, if $y> x / \exp\{ (\log x)^{1/2-\eps} \}$ and $x/y\to \infty$, then
\[
N(x,y) \sim  \frac{x \log(x/y)}{\log x}.
\]
The authors also claim in [Remark 2.10]\cite{MPP} that they can sharpen (iii) to $N(x,y) \order \frac{x\log(x/y)}{\log x}$
by taking more care of the ``singular series'' factor coming from a sieve estimate.  As we shall see below, this
is a delicate matter.

\medskip

In this paper, we determine the correct order of magnitude for $N(x,y)$ uniformly for all $x,y$
and show an asymptotic for $N(x,y)$ in most of the range (iii) of Theorem A.  As in \cite{MPP}, 
define $\alpha$ implicitly by $y=x/\exp \{ (\log x)^\alpha \}$, so that $0\le \alpha \le 1$ in the range $1\le y \le x/e$.
Near the threshold value $\alpha=\frac{1}{\log 4}$, define $\theta$ by 
\[
 \alpha = \frac{1}{\log 4} + \frac{\theta}{\sqrt{\log_2 x}}.
\]

\begin{thm}\label{main}
 We have
 (i) For $3 \le y\le x^{1-c}$, where $c>0$ is an arbitrary fixed constant,
 \[
  N(x,y) \order_c  \frac{x}{(\log y)^\delta (\log_2 y)^{1/2}}.
 \]

(ii) When $\frac{1}{\log 4} \le \alpha \le 1 - \frac{\log 2}{\log_2 x}$ (the upper bound is equivalent to $y\ge x^{1/2}$), then
\[
 N(x,y) \order \frac{x}{\max(1,\theta) (\log x)^{\delta+\alpha-1-(\log \alpha)/\log 2}}.
\]

(iii) If $x/y\to \infty$ and $\theta\to-\infty$ (in particular, $0<\a < \frac{1}{\log 4}$), then 
\[
 N(x,y) \sim \frac{x \log(x/y)}{\log x}.
\]
Uniformly in the slightly larger range $x/\exp\{(\log x)^{1/\log 4} \} \le y \le x/2$, $x\ge 10$, we have
\[
 N(x,y) \order \frac{x \log(x/y)}{\log x}.
\]
\end{thm}

\textbf{Remarks.}  In part (ii) of Theorem \ref{main}, if $\frac{1}{\log 4}<\alpha<1$ is fixed,
then $\theta \order \sqrt{\log_2 x}$.

Our proof of Theorem \ref{main} parts (ii) and (iii) refines the method offered in \cite{MPP}.  To prove the 
lower bound for part (i), we do not follow the method from \cite{MPP} (which is based on the theory of the Carmicharel $\lambda$-function),
but rather use a technique which is similar 
to that used in part (ii).  The resason that this works is described in the next section.

\medskip
\textbf{Notation}
$\omega(n)$ and $\Omega(n)$ denote the number of prime factors of $n$ and the number of prime power factors of $n$, respectively.
$\Omega(n,t)$ is the number of prime power divisors $p^a$ of $n$ with $p\le t$.
$\Omega^*(n,t)$ is the number of prime power divisors $p^a$ of $n$ with $2<p\le t$.
$P^+(n)$ and $P^-(n)$ denote the largest and smallest prime factors of $n$, respectively.

%
%
\section{Heuristic discussion}
%
%

The quantity $N(x,y)$ counts integers with a particular type of divisor, thus results about
the distribution of divisors of
integers, say from \cite[Ch. 2]{HT} or \cite{F}, may be relevant to the problem.
To bound the density of integers possessing a divisor in an interval $(y,z]$,  
the right ``measuring stick'' for the problem is sum of the densities of the integers which 
are divisible by each candidate divisor, namely the quantity $\eta := \sum_{y<d\le z} 1/d \sim \log (z/y)$.
When $\eta$ is very small, the ``events'' $d|n$ for the various $d$ are essentially independent
and the likelihood of an integer having such a divisor is about $\eta$; this independence persists below
a threshhold value of about $\eta=(\log y)^{1-\log 4}$.  As $\eta$ grows, however, these 
events become more and more dependent and when $\eta \approx 1$, the likelihood that an integer has a divisor in $(y,z]$
has dropped to about
$(\log y)^{-\delta} (\log_2 y)^{-3/2}$; moreover, the most likely integers to have such a divisor are those with
$\frac{\log_2 y}{\log 2}+O(1)$ prime divisors $\le y$, and also these prime factors must be ``nicely'' 
distributed (if not, then the divisors of $n$ are highly clustered and there is a much lower probability of
having a divisor in $(y,x]$).   When $\eta=(\log y)^{-\beta}$, with $0\le \beta \le \log 4-1$,
most integers with a divisor in $(y,z]$ have  $\Omega(n,y) = \frac{1+\beta}{\log 2}\log_2 y+O(1)$; a
heuristic explaining this may be found in \S 1.5 of \cite{F}.

For $N(x,y)$ the analogous ``measuring stick'' is the quantity $\nu = \sum_{y<p\le x} \frac{1}{p-1}$.
When $x^{1/2}<y<x/\exp\{(\log x)^{1/\log 4} \}$, $(\log (x/y))^{1-\log 4} \ll \nu \ll 1$ and this roughly 
corresponds to the ``short but not too short interval'' case for unrestricted divisors,
with $\nu$ replacing $\eta$ and $x/y$ replacing $z$ (because $x/y$ is roughly the size of the smaller factor of $n$
in this case, and that is the deciding quantity).  One might guess that the density is then given by Theorem 1
of \cite{F}, but this is not quite the case.  Because the interval $(y,z]$ in the unrestricted case is genuinely
very short, integers with typical distribution of their prime factors have very non-uniform divisor distribution (lots of tight clusters),
and this makes it highly unlikely to have a divisor in $(y,z]$.  Thus, most integers with a divisor in $(y,z]$
have an atypical prime factor distribution.  In the case of shifted prime factor divisors, the interval $(y,x]$
is very long, and this issue does not affect whether $(y,x]$ has a shifted prime divisor
and the actual liklihood is are therefore a bit larger (by roughly factor $\log_2 y$).
This also makes it much easier to obtain sharper bounds for $N(x,y)$, as delicate divisor distribution
issues do not need to be dealt with.

\medskip

The techniques of this paper may be easily adapted to obtain sharp estimates for the number of integers $n\le x$
divisible by an integer $k>y$ which comes from an arbitrary set $S$ which is ``thin'', in the sense that that
sum of reciprocals of elements of $S$ diverges very slowly like that of primes, and for which the set has 
nice distribution in arithmetic progressions (in order to apply sieve methods and obtain, e.g. analogs of the 
Timofeev bounds from the next section).

%
\section{Tools from the anatomy of integers}
%

Beginning with the work of Hardy-Ramanujan (1917), and continuing with work of Erd\H os and others in the 1930s
and beyond, it is now well-known that the prime factors of integers, viewed on a $\log\log$-scale, behave like
a Poisson process. In particular, the number of prime factors which are $\le z$ behaves roughly like a Poisson random
variable with parameter $\sim \log_2 z$ as $z\to\infty$.

\begin{lem}\label{Selberg}
 For any fixed $\delta>0$, we have uniformly for $x\ge 4$, $0\le k\le (2-\delta)\log_2 x$ that
 \[
\ssum{n\le x \\ \Omega(n)=k} \frac{1}{n} \asymp_\delta \frac{(\log_2 x)^k}{k!}  
 \]
\end{lem}
\begin{proof}
This is a corollary of a classical result of Selberg (see \cite[Theorem II.6.5]{Ten}) 
about the distribution of $\Omega(n)$.
\end{proof}

The next two lemmas are due to Hal\'asz \cite{Ha}, with an extension of Hall and Tenenbaum \cite[Theorem 08]{HT}.

\begin{lem}\label{Halup}
  Fix $\delta>0$.  Uniformly for $x\ge z\ge 3$ and $0\le m\le (2-\delta)\log_2 z$,  we have
 \begin{align*}
  \# \{n\le x : \Omega(n,z)=m \} &\ll_{\delta} \frac{x(\log_2 z)^m}{m! \log z},\\
  \ssum{n\le x \\ \Omega(n,z)=m} \frac{1}{n} &\ll_{\delta} \frac{\log x}{\log z} \, \frac{(\log_2 z)^m}{m!}.
 \end{align*}
Uniformly for $x\ge z\ge 3$ and  $0\le m\le (3-\delta)\log_2 z$,
\begin{align*}
  \# \{n\le x : \Omega^*(n,z)=m \} &\ll_{\delta} \frac{x(\log_2 z)^m}{m! \log z},\\
  \ssum{n\le x \\ \Omega^*(n,z)=m} \frac{1}{n} &\ll_{\delta} \frac{\log x}{\log z} \, \frac{(\log_2 z)^m}{m!}.
 \end{align*}
\end{lem}

\begin{lem}\label{Hallow}
  Fix $\delta>0$.  Uniformly for $x\ge z\ge 3$ and $\delta \log_2 z\le m\le (2-\delta)\log_2 z$, we have
\[
  \# \big\{ n\le x : \Omega(n,z)\in \{m,m+1\} \big\} \gg_\delta \frac{x(\log_2 z)^m}{m! \log z}.
\]
\end{lem}

The next two lemmas, due to Timofeev \cite{T}, state that the prime factors of shifted primes 
have roughly the same distribution as prime factors of integers taken as a whole.

\begin{lem}\label{Timup}
 Fix $\delta>0$.  There is some constant $c_1(\delta)$ so that uniformly for $x\ge z\ge c_1(\delta)$ and
 $0\le m\le (2-\delta)\log_2 z$, we have
 \begin{align*}
  \# \{p\le x : \Omega(p-1,z)=m \} \ll_{\delta} \frac{x(\log_2 z)^m}{m!(\log x)(\log z)}, \\
   \# \{p\le x : \Omega^*(p-1,z)=m \} \ll_{\delta} \frac{x(\log_2 z)^m}{m!(\log x)(\log z)}.
 \end{align*}
\end{lem}

\begin{proof}
 This is a special case of Theorem 2 of \cite{T}.
\end{proof}

\begin{lem}\label{Timlow}
 Fix $\delta>0$.  There is some constant $c_2(\delta)$ so that uniformly for  $x \ge z\ge c_2(\delta)$ and
 $\delta \log_2 z\le m\le (2-\delta)\log_2 z$,
 \[
  \# \{p\le x : \Omega(p-1,z)\in \{m,m+1,m+2\} \} \gg_\delta \frac{x(\log_2 z)^m}{m!(\log x)(\log z)}.
 \]
\end{lem}

\begin{proof}
 This is essentially a special case of part of Theorem 3 of \cite{T}, except that in the cited theorem it
 is stated that we must have $z\to\infty$ as $x\to \infty$.  This condition does not make sense in light of
 the uniformity claimed in the theorem, and in fact this stronger hypothesis on $z$
 (which comes into play when dealing with a set $E$ of primes, which in our application is taken to be the set of primes in $[2,z]$) 
 is never used in the proof.  Indeed, in the place where it is claimed to be needed, prior to \cite[(18)]{T}, 
 no hypothesis is needed at all on the set $E$, since $E(x/t) \le E(x)$ for any set $E$ and (18) follows 
 immediately.
\end{proof}

\begin{lem}\label{recip}
 Uniformly for $e^{2} \le z\le x$, $k\le 1.8\log_2 z$, $0\le\xi\le \frac{1}{5\log x}$, $0\le c\le 10$, we have
 \[
  \ssum{P^+(n)\le x \\ \Omega(n,z)=k} \frac{1}{n^{1-\xi}} \pfrac{n}{\phi(n)}^c \ll \frac{\log x}{\log z} \, \frac{(\log_2 z)^k}{k!}.
 \]
\end{lem}
\begin{proof}
 We follow the proof of Theorem 08 of \cite{HT} with small modifications.  Note that $\xi \le 0.1$.
Thus, since $2^{0.9} > 1.86$,  for any complex $v$ with $|v|\le 1.8$ we have
 \[
  \sum_{P^+(n)\le x} \frac{v^{\Omega(n,z)}}{n^{1-\xi}} \pfrac{n}{\phi(n)}^c = \prod_{p\le z} \(1+\frac{v}{p^{1-\xi}}
  \pfrac{p}{p-1}^c+O\pfrac{1}{p^{1.8}}\)\prod_{z<p\le x} \(1+\frac{1}{p^{1-\xi}}  \pfrac{p}{p-1}^c+O\pfrac{1}{p^{1.8}}\).
 \]
Now $\pfrac{p}{p-1}^c=1+O(1/p)$ and $p^\xi=1+O(\xi\log p)$ since $\xi\le \frac{1}{5\log x}\le \frac{1}{5\log p}$.  So 
\begin{align*}
 \sum_{P^+(n)\le x} \frac{v^{\Omega(n,z)}}{n^{1-\xi}} \pfrac{n}{\phi(n)}^c &= \prod_{p\le z}
 \(1+\frac{v}{p}+O\pfrac{\xi\log p}{p}\) \prod_{z<p\le x} \(1+\frac{1}{p}+O\pfrac{\xi\log p}{p}\) \\
 &\ll (\log z)^{\Re v}\, \frac{\log x}{\log z}.
\end{align*}
Let $r=k/\log_2 z$ and $v=re^{i\theta}$ where $0\le \theta\le 2\pi$.  Then, as in \cite{HT},
\begin{align*}
  \ssum{P^+(n)\le x \\ \Omega(n)=k} \frac{1}{n^{1-\xi}} \pfrac{n}{\phi(n)}^c &=
 \frac{1}{2\pi r^k} \int_0^{2\pi} e^{-ik\theta} \sum_{P^+(n)\le x} \frac{(re^{i\theta})^{\Omega(n,z)}}{n^{1-\xi}} 
 \pfrac{n}{\phi(n)}^c\, d\theta \\
 &\ll \frac{\log x}{\log z} \, \frac{(\log_2 z)^k}{k^k} \int_0^{2\pi} e^{k\cos\theta}\, d\theta \\
 &\ll \frac{\log x}{\log z} \, \frac{(\log_2 z)^k}{k!}. \qedhere
\end{align*}
\end{proof}

Our next tool is a hybrid of the classical theorem of Hardy-Ramanujan and the Brun-Titchmarsh inequality.

\begin{lem}[\cite{CCT}, Theorem 1.1]\label{CCT}
 Uniformly for $x>1$, $k\ge 0$, $q\in \NN$ and $(a,q)=1$ with $1\le q<x$ we have
 \[
  \ssum{n\le x \\ n\equiv a\pmod{q} \\ \omega(n)\le k} 1 \ll \frac{x}{\phi(q) \log(10x/q)} \sum_{j=0}^{k-1} \frac{(\log_2(10x/q))^{j}}{j!}.
 \]
\end{lem}

Finally, we need crude estimates for partial sums of the Poisson distribution.

\begin{lem}\label{Poisson}
 Let $v>0$ and $v^{-1/2} \le \lambda \le \frac12$.  Then
 \[
\frac{e^{-vQ(1-\lam)}}{\lam \sqrt{v}} \ll \sum_{(1-\lam)v-1/\lam \le k\le (1-\lam)v} e^{-v} \frac{v^k}{k!}  \le \sum_{k\le (1-\lam)v} e^{-v} \frac{v^k}{k!} \ll \frac{e^{-vQ(1-\lam)}}{\lam \sqrt{v}} 
 \]
and
\[
 \frac{e^{-vQ(1+\lam)}}{\lam \sqrt{v}} \ll \sum_{(1+\lam)v \le k\le (1+\lam)v+1/\lam} e^{-v} \frac{v^k}{k!}  \le \sum_{k\ge (1+\lam)v} e^{-v} \frac{v^k}{k!} \ll \frac{e^{-vQ(1+\lam)}}{\lam \sqrt{v}}, 
\]
where $Q(y)=y\log y-y+1$.
\end{lem}

\begin{proof}
 These may be found, e.g. in Norton \cite[\S 4]{Nor}.
\end{proof}

Useful corollaries of these bounds include bounds on the ``tails'' of the distribution of $\Omega(n,z)$
and $\Omega^*(n,z)$.

\begin{lem}\label{tails}
 Fix $\delta>0$ and suppose $z$ is sufficiently large in terms of $\delta$.  
 
 (i) Uniformly for $x\ge z$ and $1 \le \lam \le 1-\delta$, we have
 \[
  \# \{m\le x: \Omega(m,z) \ge (1+\lam) \log_2 z \} \ll_\delta \frac{x}{(\log z)^{Q(1+\lam)}\max(1,\lam \sqrt{\log_2 z})}.
 \] 
(ii)  Uniformly for $x\ge z$ and $1 \le w \le 2-\delta$, we have
 \[
  \# \{m\le x: \Omega^*(m,z) \ge  (1+\lam) \log_2 z \} \ll_\delta \frac{x}{(\log z)^{Q(1+\lam)}\max(1,\lam \sqrt{\log_2 z})}.
 \]
\end{lem}

\begin{proof}
 Without loss of generality, suppose $\delta<1/10$.  For (i), we have
 \begin{align*}
   \# \{m\le x: \Omega(m,z) \ge (1+\lam) \log_2 z \} &\le \sum_{(1+\lam) \log_2 z\le j\le (2-\delta/2)\log_2 z}  \# \{m\le x: \Omega(m,z) =j \}  \\
+ \quad &   \sum_{m\le x} (2-\delta/2)^{\Omega(m,z)-(2-\delta/2)\log_2 z},
 \end{align*}
the second sum being an upper bound for the number of $m\le x$ with $\Omega(m,z)>(2-\delta/2)\log_2 z$.
The terms in the first sum are estimated with the first part of Lemma \ref{Halup} together with Lemma \ref{Poisson}.
The second sum is estimated 
using standard estimates for sums of multiplicative functions, e.g. \cite[Theorem 01]{HT}, and one obtains
\[
 \sum_{m\le x} (2-\delta/2)^{\Omega(m,z)-(2-\delta/2)\log_2 z}\ll x (\log z)^{-Q(2-\delta/2)},
\]
which is smaller than the other term, since $Q(u)$ is an increasing function for $u>1$.

Part (ii) is proved similarly, using the second part of Lemma \ref{Halup}, and by breaking up the sum at $j=(3-\delta/2)\log_2 z$.
\end{proof}

%
\section{Tools from sieve methods}\label{sec:sieve}
%

\begin{lem}\label{smooth}
 We have $\# \{ n\le x : P^+(n)\le y \} \ll x e^{- 0.5 (\log x)/\log y}$
 uniformly for $x\ge y\ge 2$.
\end{lem}

\begin{proof}
 Standard.  See e.g., \cite[Theorem III.5.1]{Ten}.
\end{proof}

\begin{lem}\label{rough}
 We have $\# \{n\le x : P^-(n) > z \} \order \frac{x}{\log z}$ uniformly for $x\ge 2z\ge 4$.
 The upper bound holds uniformly for $x\ge z\ge 2$.
\end{lem}

\begin{proof}
 Standard.  Use the asymptotic formula \cite[Theorem II.6.3 and (22)]{Ten} when $x$ is large and $x/z$ is large,
 the prime number theorem for $x/z$ bounded, and Bertrand's postulate for small $x$.
\end{proof}

\begin{lem}\label{gensieve}
Suppose that $\rho$ a non-negative integer valued multiplicative function
with $\rho(p)\le \min(\kappa,p-1)$ for every prime $p$, and that for any prime $p$, $\Omega_p$ is some
set of $\rho(p)$ residue classes modulo $p$.  Then
\[
\# \{1\le n\le x : \forall p, n \!\!\mod p \not\in \Omega_p \} \ll_\kappa x \prod_{p\le x} \(1-\frac{\rho(p)}p \). 
\]
\end{lem}

\begin{proof}
 This is a standard application of Montgomery's Large Sieve, see
 e.g. \cite[Corollay I.4.6.1]{Ten}, together with an
 estimate for the denominator in the sieve bound, e.g. \cite[Lemma 4.1]{HR}.
 It does not seem to appear explicitly in the literature anywhere, to the authors knowledge.
  \end{proof}

\begin{lem}\label{primecor}
 Let $z\ge 2$, $x>2z$, and suppose $B$ and $C$ are distinct, even, positive integers.  Then
 \[
    \# \{ h\le x: P^-(h) \ge z, Bh+1 \text{ prime}\} \ll \frac{x}{(\log z)(\log x)}\, \frac{B}{\phi(B)} 
  \ll  \frac{x \log_2 (2B) }{(\log z)(\log x)}.
 \]
and
 \begin{align*}
  \# \{ h\le x: P^-(h) \ge z, Bh+1 \text{ prime}, Ch+1 \text{ prime} \} &\ll \frac{x}{(\log z)(\log^2 x)}
  \prod_{p|BC(B-C)} \frac{p}{p-1} \prod_{p|(B,C)} \frac{p}{p-1} \\
  &\ll  \frac{x (\log_2 BC)^2}{(\log z)(\log^2 x)}.
 \end{align*}
\end{lem}

\begin{proof}
Completely routine exercise using Lemma \ref{gensieve}.  If $z=2$ or $\log z \gg \log x$, these
follow from classical literature, e.g. \cite[Theorem 2.2]{HR}.
\end{proof}

%
%
\section{Proof of Theorem \ref{main}: upper bounds}\label{sec:upper}
%
%

The upper bound in part (i) is proven in \cite[Theorem 1.2]{MPP}.

The upper bound in part (iii) is very easy.  Mertens' theorem implies that
\[
 N(x,y) \le \sum_{y<p\le x} \frac{x}{p-1} =x(\log_2 x-\log_2 y + O(1/\log y)).
\]
If $x/y\to \infty$, the right side is $\sim x \frac{\log(x/y)}{\log x}$.  In the larger range $y\le x/2$, 
the right side is $O(x \frac{\log(x/y)}{\log x})$.

Finally, we prove part (ii).
Let  $z=x/y$, $\gamma=\frac{1}{\a \log 4}$ and $w=\fl{\gamma \log_2 z}$.  The hypotheses on $\alpha$
imply that $w \le \log_2 z$.
  Consider first integers $n\le x$ with $\Omega^*(n,z)>2w$.  By Lemma \ref{tails},
the number of such $n$ is 
\[
\ll \frac{x}{(\log z)^{Q(2\gamma)}\sqrt{\log_2 z}}  = \frac{x}{(\log x)^{\delta+\a-1-\log \a/\log 2}\sqrt{\log_2 z}}.
\]
Next, consider integers of the form $n=(p-1)m$ with $m\le z$, $\Omega^*(p-1,z)=i$ and $\Omega^*(m)=j$, where $i+j\le 2w$.
With $i$ and $j$ fixed, we may use Lemma \ref{Timup}, provided that $i\le 1.99\log_2 z$, together with Lemma \ref{Halup}, to bound the number of such $n$ by
\[
 \ll \ssum{m\le z \\ \Omega^*(m)=j}  \frac{x(\log_2 z)^i}{i! m (\log x)(\log z)} \ll\frac{x(\log_2 z)^{i+j}}{i! j! (\log x)(\log z)}.
\]
By Lemma \ref{Poisson}, the total number of integers counted is 
\begin{align*}
 &\ll \frac{x}{(\log x)(\log z)} \sum_{i+j\le 2w} \frac{(\log_2 z)^{i+j}}{i!j!}\\
 &\ll  \frac{x}{(\log x)(\log z)} \sum_{h\le 2w} \frac{(2\log_2 z)^h}{h!} \\
 &\ll \frac{x}{(\log x)^{\delta+\a-1-\log \a/\log 2}\max(1,\theta)}.
\end{align*}
If $i\ge 1.99\log_2 z$, then $j\le 0.01\log_2 z$.  The number of such integers is bounded above by
\[
 \ssum{m\le z \\ \Omega^*(m,z)\le 0.01\log_2 z} \pi(x/m) \ll \frac{x}{\log x} \ssum{m\le z \\ \Omega^*(m,z)\le 0.01\log_2 z} \frac{1}{m}
 \ll \frac{x}{\log x} (\log z)^{0.01+0.01\log 100} \ll \frac{x}{(\log x)^{0.9}},
\]
using Lemma \ref{Halup}, which is much smaller than the bound for the other cases.
This completes the proof of the upper bound in part (ii).

%
%
\section{Proof of Theorem \ref{main} (iii) lower bound when $\theta\to -\infty$}
%
%

Here we prove the lower bound claim in part (iii) of the theorem, except in the case where $\theta$ is positive
and bounded.  We begin with a Lemma, which is similar to Lemma \ref{tails}.

\begin{lem}\label{sumphi}
 Uniformly for $z$ sufficiently large and $0 \le \lambda \le 0.7$, we have
 \[
  \ssum{P^+(m)\le z \\ \Omega(m)>(1+\lam)\log_2 z} \frac{1}{\phi(m)} \ll  \frac{(\log z)^{1-Q(1+\lam)}}{\max(1,\lam\sqrt{\log_2 z})}.
 \]
\end{lem}

\begin{proof}
Let $w=(1+\lam)\log_2 z$.
 We use Lemma \ref{recip} to take care of the summands with $w \le \Omega(m) \le 1.8\log_2 z$ and a simple 
 ``Rankin trick'' for the rest, as in the proof of Lemma \ref{recip}.  We obtain
 \[
   \ssum{P^+(m)\le z \\ \Omega(m)>w} \frac{1}{\phi(m)} \ll \sum_{w\le j\le 1.8\log_2 z} \frac{(\log_2 z)^j}{j!} 
   +\sum_{P^+(m)\le z} \frac{1.8^{\Omega(m)-1.8\log_2 z}}{\phi(m)}.
 \]
Use Lemma \ref{Poisson} for the first sum.  The second sum equals
\[
 (\log z)^{-1.8\log 1.8} \prod_{p\le z} \(1 + \frac{1.8}{p-1} + \frac{1.8^2}{p(p-1)}+\cdots\) \ll (\log z)^{1-Q(1.8)},
\]
which is smaller than the bound claimed.
\end{proof}

Let $z=x/y$.
First assume that $\alpha \le 1/3$.  Let $r(n)$ be the number of ways to write 
$n=(p-1)m$ where $y<p\le x$ is prime and $m$ is any integer.  Note that $m\le z$ is very small.  
By the upper bound calculation,
 $M_1 := \sum_{n\le x} r(n) \sim x \frac{\log z}{\log x}$ if $z\to \infty$ 
 (and $M_1\gg  x \frac{\log z}{\log x}$ 
 in the larger range $y\le x/2$).
The quantity $M_2'=\sum_{n\le x} r(n)^2-r(n)$ counts solutions of $(p_1-1)m_1=(p_2-1)m_2$ with $p_1\ne p_2$.
Put $a=(m_1,m_2)$, $m_1=ab$, $m_2=ac$, $g=(p_1-1,p_2-1)$, so that $p_1-1=gc$ and $p_2-1=gb$.
Note that $abc\le z^2 \ll x^{1/10}$.
By Lemma \ref{primecor}, given $a,b,c$ the number
of choices for $g$ is $O(\frac{x(\log_2 x)^{2}}{abc\log^2 x})$.  Hence
\[
 M_2' \ll \frac{x(\log_2 x)^2}{\log^2 x} \sum_{a,b,c\le z} \frac{1}{abc} \ll \frac{x(\log_2 x)^2}{(\log x)^{2-3\a}} = 
 o(M_1).
\]
By simple inclusion-exclusion, $N(x,y) \ge M_1 - M_2' \sim  x \frac{\log z}{\log x}$ if $z\to \infty$,
and in the larger range $y\le x/2$ we have  $N(x,y) \ge M_1-M_2' \gg x \frac{\log z}{\log x}$.

Now assume that $\alpha \ge 1/3$.   It follows easily from $\theta\to-\infty$ that $z\to\infty$ as well.
Let $r(n)$ be the number of ways to write $n=(p-1)m$, with $p$ prime, $p>y$,
and  $\max(\Omega(m),\Omega(p-1,z)) \le w$, where  $w = \fl{\log_2 z - (\theta/2)\sqrt{\log_2 z}}$.
The hypotheses on $\alpha$ imply that $\log_2 z \le w \le 1.7\log_2 z$.
We have
\be\label{M1-iii}
 M_1 = \sum_{m\le z} \big( \pi(x/m)-\#\{p\le x/m:\Omega(p-1,z)>w\} \big) - O\Bigg( \ssum{m\le z \\
 \Omega(m)>w} \pi(x/m) \Bigg).
\ee
Applying Lemma \ref{sumphi}, we quickly find that the big-$O$ term in \eqref{M1-iii} is
\[
 \ll \frac{x}{\log x} \ssum{m\le z \\ \Omega(m) > w} \frac{1}{m} \ll \frac{x(\log z)^{1-Q(w/\log_2 z)}}{(-\theta) \log x}
  = o\pfrac{x\log z}{\log x}.
\]
Next, consider a prime $p\le x/m$ with $\Omega(p-1,z)>w$.   The number of primes with $P^+(p-1)\le z$
is, by Lemma \ref{smooth},  $O(x/(m\log^{10}x))$.  If $P^+(p-1)>z$, let $k$  
be the largest factor of $p-1$ which is composed only 
of primes $\le z$, so that $k\le x/mz$ and $\Omega(k)>w$.  By Lemma \ref{primecor},
the number of such primes $p$ is, for a given $k$, $O(\frac{x}{m\phi(k)\log x \log z})$.  Thus the total 
number of such primes is, using Lemma \ref{sumphi}, bounded above by
\[
 \ll \frac{x}{m\log x(\log z)} \ssum{P^+(k)\le z \\ \Omega(k)>w} \frac{1}{\phi(k)} \ll \frac{x(\log z)^{-Q(w/\log_2 z)}}{(-\theta) m\log x}
 = o\pfrac{x}{m\log x}.
\]
We also have that 
\[
\sum_{m\le z} \pi(x/m) \sim \frac{x}{\log x} \sum_{m\le z} \frac{1}{m} \sim \frac{x\log z}{\log x},
\]
and therefore conclude from \eqref{M1-iii} that
\be\label{M1-asym}
M_1 \sim \frac{x\log z}{\log x} \qquad (\theta\to -\infty).
\ee

Arguing as in \cite{MPP}, $M_2':=\sum_{n\le x} r(n)^2-r(n)$ counts the number of solutions
of $(p_1-1)m_1=(p_2-1)m_2$ with $\Omega(p_i-1,z)\le w$,  $\Omega(m_i)\le w$ for $i=1,2$, and $p_1\ne p_2$.
We may assume $p_1<p_2$.
Again put $a=(m_1,m_2)$, $m_1=ab$, $m_2=ac$, $g=(p_1-1,p_2-1)$, so that $p_1-1=gc$ and $p_2-1=gb$.
Let $g=dh$, where $P^+(d)\le z < P^-(h)$.  Observe that $d \le z^{\Omega(d)} \le z^w \le x^{1/10}$.
Given $a,b,c,d$, we bound the number of $h$ with $hcd+1$ and $hbd+1$ both prime using Lemma \ref{primecor}, and get
\[
M_2' \ll \frac{x(\log_2 z)^{2}}{(\log^2 x)\log z} \ssum{a,b,c,d\le z \\ \Omega(abcd)\le 2w} \frac{1}{abcd}.
\]
Since
\begin{align*}
 \ssum{a,b,c,d\le z \\ \Omega(abcd)\le 2w} \frac{1}{abcd} &\le \sum_{a,b,c,d\le z} \frac{2^{2w-\Omega(abcd)}}{abcd} 
 \le 2^{2w} \prod_{p\le z} \(1-\frac{1}{2p} \)^{-4} \\
 &\ll 2^{2w} (\log z)^2 \ll (\log z)^{2+\log 4} \exp\{(\log 4) (-\theta/2) \sqrt{\log_2 z} \},
\end{align*}
we get
\begin{align*}
 M_2' &\ll \frac{x\log z}{\log x} \; \frac{(\log z)^{\log 4}}{\log x} (\log_2 z)^{2} \exp\{(\log 4)(-\theta/2)\sqrt{\log_2 z} \} \\
&=\frac{x\log z}{\log x} \; (\log_2 z)^{2} \exp \left\{ -\theta (\log 4) \( -\sqrt{\log_2 x} + \tfrac12\sqrt{\log_2 z} \) \right\}
 =o(M_1) \qquad (\theta\to-\infty).
\end{align*}
The theorem now follows upon comparing with \eqref{M1-asym}.

%
%
\section{Proof of Theorem \ref{main} lower bounds (i), (ii), and (iii) when $-\theta$ is bounded}
%
%

In the proof, we will need to bound sums of the type
\[
 S(z,Y;w;\xi) := \ssum{a,b,c\le z \\ P^+(d)\le Y \\ \Omega(abcd)\le w \\ b>c} \frac{1}{abcd^{1-\xi}} 
 \pfrac{d}{\phi(d)}^2 \frac{b}{\phi(b)}\cdot \frac{c}{\phi(c)}\cdot \frac{b-c}{\phi(b-c)}.
\]
The only complicated part to take care of is the fraction $\frac{b-c}{\phi(b-c)}$.

\begin{lem}\label{bigsum}
 Suppose that $z\ge e^{3}$, $2\le Y\le z$, $1\le w\le 1.5\log_2 z$ and $0\le \xi \le \frac{1}{10\log Y}$.
 (i) If $Y \le \exp \{ (\log z)^{0.99} \}$, then
 \[
   S(z,Y;w;\xi) \ll (\log z)^5.
 \]
 (ii) If $Y \ge \exp \{ (\log z)^{0.99} \}$, then
 \[
  S(z,Y;w;\xi) \ll \frac{(4\log_2 z)^w}{w!}.
 \]
\end{lem}

\begin{proof}
Part (i) is immediate from the elementary bounds $n/\phi(n) \ll \log_2 P^+(n)$, $\sum_{n\le z} 1/n \ll \log z$ and
\[
 \sum_{P^+(d)\le Y} \frac{1}{d^{1-\xi}} \ll  \prod_{p\le Y} \(1 + \frac{p^\xi}{p} \)
 \le \exp \left\{ \sum_{p\le Y} \frac{1+O(\xi\log p)}{p} \right\} \ll \log Y,
\]
since $\xi \le \frac{1}{10\log Y} \le \frac{1}{10\log p}$ for all $p\le Y$.

For part (ii), first apply Cauchy's inequality, and get that  $S(z,Y;w;\xi) \le S_1^{1/2} S_2^{1/2}$, where
\[
S_1=\sum_{a,b,c,d} \frac{1}{abcd^{1-2\xi}} \pfrac{d}{\phi(d)}^4 \pfrac{b}{\phi(b)}^2 \pfrac{c}{\phi(c)}^2, \quad 
S_2 = \sum_{a,b,c,d} \frac{1}{abcd}\pfrac{b-c}{\phi(b-c)}^2,
\]
where in each sum we have the same conditions on $a,b,c,d$.
We may quickly deal with $S_1$ using Lemma \ref{recip} repeatedly.
\be\label{S1}
 S_1 \ll \sum_{r+s+t+u\le w} \frac{(\log_2 z)^{r+s+t}(\log_2 Y)^u}{r!s!t!u!} \le \sum_{j\le w} \frac{(4\log_2 z)^j}{j!} \ll 
 \frac{(4\log_2 z)^w}{w!}.
\ee
where we  used the lower bound on $Y$ which implies that $u\le w \le 1.5\log_2 z \le 1.8\log_2 Y$ and $Y\ge e^2$.

For $S_2$, write $(\frac{f}{\phi(f)})^2 = \sum_{l|f} g(l)$, where $g$ is multiplicative, supported on squarefree numbers and
$g(p)=\frac{2p-1}{(p-1)^2}$ for primes $p$.  Let $l_0 = \lfloor \log^5 z \rfloor$.  Recalling that $(b,c)=1$, we then have
\begin{align*}
 S_2 & = \sum_{l} g(l) \ssum{a,b,c,d \\ b>c, l|(b-c)} \frac{1}{abcd} \\
 &=\sum_{l\le l_0} g(l) \sum_{r+s+t+u\le w} \ssum{a\le z\\\Omega(a)=r} \frac{1}{a} \ssum{P^+(d)\le Y\\ \Omega(d)=s} \frac{1}{d}
\ssum{c\le z\\\Omega(c)=u \\ (c,l)=1} \frac{1}{c} \; \ssum{c<b\le z \\ \Omega(b)=t \\ b\equiv c \pmod{l}} \frac{1}{b} + E,
 \end{align*}
where the ``error term'' $E$ satisfies
\begin{align*}
 E &\le \sum_{l>l_0} g(l) \ssum{a,b,c\le z, P^+(d)\le z \\ b>c, l|(b-c)} \frac{1}{abcd} \\
 &\ll \sum_{l>l_0} \frac{3^{\omega(l)}}{l} (\log z)^2 \sum_{c\le z} \frac{1}{c} \ssum{c<b\le z\\ b\equiv c\pmod{l}} \frac{1}{b} \\
 &\ll  \sum_{l>l_0} \frac{3^{\omega(l)}}{l} \frac{(\log z)^4}{l} \ll \frac{(\log z)^4}{l_0^{0.99}} \ll 1.
\end{align*}
By Lemma \ref{CCT} and partial summation (since $\Omega(c)=u$ implies that $\omega(c)\le u$),
\[
\ssum{c<b\le z \\ \Omega(b)=t \\ b\equiv c \pmod{l}} \frac{1}{b} \ll \int_{c+l}^z \frac{1}{\phi(l)s\log(10s/l)}\sum_{j=0}^{t-1} \frac{(\log_2 (10s/l))^j}{j!}\, ds
\ll \frac{1}{\phi(l)} \sum_{j=0}^t \frac{(\log_2 y)^j}{j!}.
\]
Applying Lemma \ref{Halup} to the sums over $a,b$ and $d$ in $S_2$, we obtain that
\begin{align*}
 S_2 &\ll \sum_{l\le l_0} \frac{g(l)}{\phi(l)} \sum_{r+s+t+u\le w} \sum_{j=0}^t \frac{(\log_2 y)^{r+s+u+j}}{r!s!u!j!}\\
 &\ll \sum_{r+s+u+j \le w} \frac{(\log_2 y)^{r+s+u+j}}{r!s!u!j!}(w-r-s-u-j+1) \\
 &= \sum_{v\le w} (w-v+1) \frac{(4\log_2 y)^v}{v!}.
\end{align*}
Since $w-v+1 \ll 2^{w-v}$, we quickly arrive at
\be\label{S2}
 S_2 \ll \frac{(4\log_2 z)^w}{w!} + 1 \ll \frac{(4\log_2 z)^w}{w!}.
\ee
Combining \eqref{S1} and \eqref{S2} gives the lemma.
\end{proof}

Now we prove the lower bound in Theorem \ref{main}.
Write $z=\min(y,x/y)$, $\gamma=\min(1,\frac{1}{\a \log 4})$ and put $w=2\fl{\gamma \log_2 z}$.
Notice that $\gamma=\frac{1}{\a \log 4}$ unless $\theta\le 0$ and we are in case (iii).

We may assume without loss of generality that (a) $y\le x^{1/11}$ or that (b) $y\ge x^{10/11}$, for if 
$x^{1/11} < y < x^{10/11}$, we have $N(x,y) \ge N(x,x^{10/11})$ and the result follows from the lower bound
for the case $y=x^{10/11}$.  Consequently, in either case we have $z\le x^{1/11}$.
We may also assume that $x$ and $y$ are sufficiently large so that $z\ge e^3$.

Let $r(n)$ denote the number of ways to factor $n$ as $n=(p-1)m$, where $p$ is prime, $y<p<y^{1.1}$, 
$\Omega((p-1)m,z)\le w$, and furthermore we have in case (a) that $m=kh$ with $k\le y^{1/10}$ and 
$P^-(h)>y^{1.1}$.

By Cauchy's inequality,
\be\label{Cauchy}
 N(x,y) \ge \# \{ n\le x: r(n)>0 \} \ge \frac{M_1^2}{M_2}, \quad M_1:=\sum_{n\le x} r(n), \quad M_2:=\sum_{n\le x} r(n)^2.
\ee
First we bound $M_1$ from below.  Start with case (a).
Given $p$ and $k$, we have $(p-1)k \le y^{1.2} \le x^{0.2}$, so by Lemma \ref{rough},
the number of possible choices for $h$ is $\gg \frac{x}{pk \log y}$.  Hence
\[
 M_1 \gg \frac{x}{\log y} \ssum{y<p\le y^{1.1} \\ k\le y^{1/10} \\ \Omega(k(p-1))\le w} \frac{1}{kp}.
\] 
We note that $\a = 1-O(1/\log_2 x)$ and $w=\frac{\log_2 y}{\log 2}+O(1)$.
Consider numbers with $\Omega(k)=m_1$, $\Omega(p-1)=m_2$
and $m_1+m_2\in \{w-2,w-1,w\}$.  With $m_1 \le 0.9w-2$ fixed, Lemma \ref{Timlow} implies that
\[
 \ssum{y<p\le y^{1.1} \\ w-2-m_1\le  \Omega(p-1) \le w-m_1}  \frac{1}{p} \gg \frac{(\log_2 y)^{w-2-m_1}}{(w-2-m_1)! \log y}.
\]
By Lemma \ref{Selberg},
\[
 \ssum{k\le y^{1/10} \\  \Omega(k) = m_1} \frac{1}{k} \gg \frac{(\log_2 y)^{m_1}}{m_1!}
\]
uniformly in $m_1$.  
Putting these bounds together and summing on $m_1$, we obtain
\be\label{M1-i}
\begin{split}
M_1 &\gg \frac{x(\log_2 y)^{w-2}}{\log^2 y} \sum_{1\le m_1\le 0.9w-2} \frac{1}{m_1! (w-2-m_1)!} \\
&\gg \frac{x (\log_2 y)^{w-2} 2^{w-2}}{(\log^2 y) (w-2)!}
\gg \frac{x}{(\log y)^{\delta} \sqrt{\log_2 y}}.
\end{split}
\ee

In case (b), we similarly use  Lemmas \ref{Selberg} and \ref{Timlow} to bound separately the number of $n$ with
$\Omega(m)=j$ and $\Omega(p-1,z)\in\{k-2,k-1,k\}$  with $j+k\le w$.  This gives
\begin{align*}
 M_1 &\ge \frac13 \ssum{j+k\le w \\ 0.1w \le k\le 0.9w} \ssum{m\le z \\ \Omega(m)= j} \#\{p\le x/m: \Omega(p-1,z)\in \{k-2,k-1,k\} \} \\
 &\gg \ssum{j+k\le w \\ 0.1w \le k\le 0.9w} \frac{x (\log_2 z)^{j+k-2}}{(\log x) (\log z) j! (k-2)!}.
\end{align*}
Next, gather together the summands with $j+k=l$ for fixed $l\ge w/2$.  We obtain
\be\label{M1-ii}
\begin{split}
M_1 &\gg  \frac{x}{(\log x) (\log z)} \sum_{w/2\le l\le w-2} \frac{(\log_2 z)^l}{l!} \ssum{k\le l \\ 0.1w \le k\le 0.9w} \binom{l}{k-2} \\
&\gg  \frac{x}{(\log x) (\log z)} \sum_{w/2\le l\le w-2} \frac{(2\log_2 z)^l}{l!} \\
&\gg \frac{x}{\log x (\log z)} \frac{e^{(2\log_2 z)(1- Q(\gamma))}}{\max(1,\theta)} \\
&= \frac{x}{\max(1,\theta)(\log x)(\log z)^{1+2\gamma\log\gamma-2\gamma}}.
\end{split}
\ee

We next bound from above the quantity $M_2'=M_2-M_1$.  In case (a), $M_2'$ 
 counts the number of solutions of
\[
 (p_1-1)k_1 h = (p_2-1)k_2 h \le x,
\]
with $p_j$ prime, $y<p_j<y^{1.1}$,
$k_j<y^{1/10}$, $\Omega(k_j)+\Omega(p_j-1) \le w$\;  $(j=1,2)$,
$P^-(h)>y^{1.1}$ and $p_1\ne p_2$.  We may assume that $p_1<p_2$.  Given $p_1,p_2,k_1,$ and $k_2$, 
there are 
$$O\pfrac{x}{(p_1-1) k_1 \log y}$$ choices for $h$ by Lemma \ref{rough}.
Write $a=(k_1,k_2)$, put $k_1=ab$, $k_2=ac$,
$g=(p_1-1,p_2-1)$ so that $p_1-1=cg$, $p_2-1=bg$.   Note that $g> y^{9/10}$ and $b> c$.  
Let $t=P^+(g)$, $g= t d$.
Then $(p_1-1)k_1=abcd t$.  Suppose that $T \le t < 2T$, where $T$ is a power of 2.
by Lemma \ref{primecor}, if $a,b,c,d,T$ are fixed, the number of $t$ such that $t$, $cdt+1$ and
$bdt+1$ are all prime is bounded above by
\[
 \ll \frac{T}{\log^3 T} \pfrac{d}{\phi(d)}^2 \frac{b}{\phi(b)}\cdot \frac{c}{\phi(c)}\cdot \frac{b-c}{\phi(b-c)}.
\]
If $T\ge y^{0.1}$, set $\xi=0$, and otherwise let $\xi=\frac{1}{10\log(2T)}$.
In the latter case  $d\ge y^{0.8}$ and thus in either case we have
\[
 \frac{1}{d} \le \frac{1}{y^{0.8\xi}} \cdot \frac{1}{d^{1-\xi}}.
\]
Hence, 
\[
M_2' \ll \frac{x}{\log y} \sum_{T=2^j\le y^{1.1}} \frac{1}{y^{0.8\xi}\log^3 T} S(y^{1.1};2T;w;\xi).
 \]
If $T < \exp \{ (\log y)^{0.99} \}$, then $y^{0.8\xi} > (\log y)^{100}$, hence by Lemma \ref{bigsum},
these summands contribute  $O(x/\log^{80} y)$ to
 the above right side.  For each $T$ satisfying
$T \ge \exp \{ (\log y)^{0.99} \}$, we see from Lemma  \ref{bigsum} that
\[
\frac{S(y^{1.1};2T;w;\xi)}{y^{0.8\xi}\log^3 T} \ll \frac{x\log^{2-\delta} y}{\sqrt{\log_2 y}} \; \frac{1}{(\log^3 T) e^{0.08 \log y/\log (2T)}}.
\]
Summing over $T$ which are powers of two, we get
\be\label{M2-i}
 M_2' \ll \frac{x}{(\log y)^\delta \sqrt{\log_2 y}}.
\ee

In case (b), $M_2'$ equals twice the number of solutions of the equation
\[
(p_1-1)m_1=(p_2-1)m_2 \le x,
\]
with $p_i>y$, $p_1 < p_2$, $m_i\le z=x/y$, and $\Omega((p_i-1)m_i,z)\le w$ for $i=1,2$. 
As in case (i), we write $a=(m_1,m_2)$, $m_1=ab$, $m_2=ac$, and note that $b>c$.
Also write $(p_1-1,p_2-1)=dh$, where $P^+(d)\le z < P^-(h)$.  Then $p_1-1=cdh$ and $p_2-1=bdh$.
There are two cases to consider:
 $d\le \sqrt{x}$ and $d>\sqrt{x}$.
If $d\le \sqrt{x}$, we have $abcd \le x^{1/2} z^2 \le x^{15/22}$.  Using Lemma \ref{primecor} to bound the number of 
possible $h$ for a given quadruple $(a,b,c,d)$, and then applying Lemma \ref{bigsum} and Stirling's formula, 
we find that
the number of solutions in this case is bounded by
\be\label{smalld}
\begin{split}
  &\ll \frac{x}{(\log z)(\log^2 x)} S(z;z;w;0) \\
  &\ll \frac{x}{(\log z)(\log^2 x)}\, \frac{(4\log_2 z)^w}{w!} \\
  &\ll \frac{x}{(\log x)^2 (\log z)^{1+2\g\log \g-2\g-2\g\log 2}\sqrt{\log_2 z}}.
\end{split}
\ee
Now assume $d>\sqrt{x}$ and let $t=P^+(d)$, $d=d't$.  We further subdivide into two subcases.  If 
$abcd'h\le x^{3/4}$, then by Lemma \ref{primecor}, for each quintuple $(a,b,c,d',h)$ the number of 
possible $t\le \frac{x}{abcd'h}$ with $t$, $tbd'h+1$ and $tcd'h+1$ all prime (with $b>c$) is
\[
 \ll \frac{x}{abcd' h\log^3 x}\, \pfrac{d'}{\phi(d')}^2 \frac{b}{\phi(b)}\cdot \frac{c}{\phi(c)}\cdot \frac{b-c}{\phi(b-c)}.
\]
Summing over al possible $a,b,c,d',h$ we see that the above is $\ll \frac{x}{(\log^2 x)(\log z)} S(z;z;w;0)$
and we get the same bound as in \eqref{smalld} for the number of solutions.
When $abcd'h>x^{3/4}$, we note that $t\ge P^+(d')$ and thus $abcd'h \le x/P^+(d')$.  
By Lemma \ref{primecor}, for each quintuple $(a,b,c,d',h)$ the number of 
possible $t$ is
\[
 \ll \frac{x}{abcd' h\log^3 P^+(d')}\, \pfrac{d'}{\phi(d')}^2 \frac{b}{\phi(b)}\cdot \frac{c}{\phi(c)}\cdot \frac{b-c}{\phi(b-c)}.
\]
Suppose that $V<P^+(d')\le V^2$, where $V$ is of the form $V=x^{1/2^l}$ for some positive ineger $l$.
Put $\xi=\frac{1}{10\log (V^2)}$.  Since $t<x^{1/4}$,  $d'>x^{1/4}$  and it follows that
\[
 \frac{1}{d'} \le \frac{1}{x^{\xi/4} (d')^{1-\xi}}.
\]
By Mertens estimate,
\[
 \ssum{h\le x \\ P^-(h)>z} \frac{1}{h} \ll \frac{\log x}{\log z}.
\]
Summing over all possible $a,b,c,d'$, we find that the total number of solutions counted in this subcase is 
at most
\[
 \ll \frac{x\log x}{\log z} \sum_{V=x^{1/2^l}\le z^2} \frac{1}{x^{\xi/4} \log^3 V} S(z;V^2;w;\xi).
\]
When $V\le \exp \{ (\log z)^{0.99} \}$, $x^{\xi/4} > (\log x)^{100}$ and the number of solutions is 
$O(x/\log^{80} x)$ by Lemma \ref{bigsum}.  Otherwise, using Lemma \ref{bigsum}, the number is bounded above by
\begin{align*}
&\ll \frac{x\log x}{\log z} \sum_{V=x^{1/2^l}\le z^2} \frac{1}{x^{\xi/4}\log^3 V} \frac{(4\log_2 z)^w}{w!}\\
&\ll \frac{x}{(\log x)^2 (\log z)^{1+2\g\log \g-2\g-2\g\log 2}\sqrt{\log_2 z}} 
\sum_{l\ge 1} \frac{2^{3l}}{\exp\{2^{l-7} \}} \\
&\ll \frac{x}{(\log x)^2 (\log z)^{1+2\g\log \g-2\g-2\g\log 2}\sqrt{\log_2 z}}.
\end{align*}
Recalling the bound \eqref{smalld} for the number of solutions with $d \le \sqrt{x}$, we find that
\be\label{M2-ii}
 M_2' \ll \frac{x}{(\log x)^2 (\log z)^{1+2\g\log \g-2\g-2\g\log 2}\sqrt{\log_2 z}} \le
 \frac{x}{(\log x) (\log z)^{1+2\g\log \g-2\g}\sqrt{\log_2 z}}.
\ee
since $\g \le \frac{1}{\a \log 4}$.

Inserting \eqref{M1-i} and \eqref{M2-i} into \eqref{Cauchy} gives the desired bound for part (i).
Inserting \eqref{M1-ii} and \eqref{M2-ii} into \eqref{Cauchy} gives the desired bound for part (ii),
and also handles the case when $-\theta$ is bounded in part (iii).


\end{document}